\documentclass{article}
\usepackage{amsmath,bm,colortbl,graphicx,mathrsfs,afterpage,amssymb,url,cite,verbatim,amsthm}
\RequirePackage[colorlinks,citecolor=blue,urlcolor=blue]{hyperref}
\theoremstyle{plain}
\pdfoutput=1
\newtheorem{theorem}{Theorem}
\newtheorem{lemma}[theorem]{Lemma}
\newtheorem{corollary}[theorem]{Corollary}
\theoremstyle{remark}
\allowdisplaybreaks

\begin{document}

\title{Parrondo games with spatial dependence, II}
\author{S. N. Ethier\\
\begin{small}University of Utah\end{small}\\
\begin{small}Department of Mathematics\end{small}\\
\begin{small}155 S. 1400 E., JWB 233\end{small}\\
\begin{small}Salt Lake City, UT 84112, USA\end{small}\\
\begin{small}e-mail: \url{ethier@math.utah.edu}\end{small}
\and
Jiyeon Lee\\
\begin{small}Yeungnam University\end{small}\\
\begin{small}Department of Statistics\end{small}\\
\begin{small}214-1 Daedong, Kyeongsan\end{small}\\
\begin{small}Kyeongbuk 712-749, South Korea\end{small}\\
\begin{small}e-mail: \url{leejy@yu.ac.kr}\end{small}}
\date{}

\maketitle

\begin{abstract}
\noindent Let game $B$ be Toral's cooperative Parrondo game with (one-dimensional) spatial dependence, parameterized by $N\ge3$ and $p_0,p_1,p_2,p_3\in[0,1]$, and let game $A$ be the special case $p_0=p_1=p_2=p_3=1/2$.  In previous work we investigated $\mu_B$ and $\mu_{(1/2,1/2)}$,  the mean profits per turn to the ensemble of $N$ players always playing game $B$ and always playing the randomly mixed game $(1/2)(A+B)$.  These means were computable for $3\le N\le19$, at least, and appeared to converge as $N\to\infty$, suggesting that the Parrondo region (i.e., the region in which $\mu_B\le0$ and $\mu_{(1/2,1/2)}>0$) has nonzero volume in the limit.  The convergence was established under certain conditions, and the limits were expressed in terms of a parameterized spin system on the one-dimensional integer lattice.  In this paper we replace the random mixture with the nonrandom periodic pattern $A^r B^s$, where $r$ and $s$ are positive integers.  We show that $\mu_{[r,s]}$, the mean profit per turn to the ensemble of $N$ players repeatedly playing the pattern $A^r B^s$,
is computable for $3\le N\le 18$ and $r+s\le4$, at least, and appears to converge as $N\to\infty$, albeit more slowly than in the random-mixture case.  Again this suggests that the Parrondo region ($\mu_B\le0$ and $\mu_{[r,s]}>0$) has nonzero volume in the limit.  Moreover, we can prove this convergence under certain conditions and identify the limits.
\end{abstract}

\section{Introduction}

Toral \cite{T01} introduced what he called \textit{cooperative Parrondo games}, in which there are $N\ge3$ players labeled from 1 to $N$ and arranged in a circle in clockwise order.  At each turn, one player is chosen at random to play.  Call him player $i$.  He plays either game $A$ or game $B$.  In game $A$ he tosses a fair coin.  In game $B$, he tosses a $p_0$-coin (i.e., $p_0$ is the probability of heads) if his neighbors $i-1$ and $i+1$ are both losers, a $p_1$-coin if $i-1$ is a loser and $i+1$ is a winner, a $p_2$-coin if $i-1$ is a winner and $i+1$ is a loser, and a $p_3$-coin if $i-1$ and $i+1$ are both winners.  (Because of the circular arrangement, player 0 is player $N$ and player $N+1$ is player 1.) A player's status as winner or loser depends on the result of his most recent game.  The player of either game wins one unit with heads and loses one unit with tails.  Under these assumptions, the model has an integer parameter, $N$, and four probability parameters, $p_0$, $p_1$, $p_2$, and $p_3$.  See Figure~\ref{diagram}.

\begin{figure}\label{diagram}
\caption{Toral's spatially dependent, or \textit{cooperative}, Parrondo games, with parameters $N\ge3$ and $p_0,p_1,p_2,p_3\in[0,1]$.  (A player's status as winner or loser depends on the result of his most recent game.  Players are labeled from 1 to $N$; player 0 is player $N$ and player $N+1$ is player 1.  A $p$-coin is one for which the probability of heads is $p$.)\medskip}
\setlength{\unitlength}{1.cm}
\begin{picture}(12,6.8)
\put(2.25,6){\framebox(1.5,.5){game $A$}}
\put(8.25,6){\framebox(1.5,.5){game $B$}}
\put(0.3,5){\framebox(5.4,.5){choose a player, say $i$, at random}}
\put(3,5){\vector(0,-1){.5}}
\put(6.3,5){\framebox(5.4,.5){choose a player, say $i$, at random}}
\put(9,5){\vector(0,-1){.5}}
\put(1.4,4){\framebox(3.2,.5){$i$ tosses a fair coin}}
\put(9,4){\vector(0,-1){.5}}
\put(6.2,4){\framebox(5.6,.5){check status of players $i-1$, $i+1$}}
\put(6.1,1.5){\framebox(5.8,2){\begin{minipage}{.44\linewidth}
{$\bullet$ loser, loser: $i$ tosses a $p_0$-coin\\
$\bullet$ loser, winner: $i$ tosses a $p_1$-coin\\
$\bullet$ winner, loser: $i$ tosses a $p_2$-coin\\
$\bullet$ winner, winner: $i$ tosses a $p_3$-coin}\end{minipage}}}
\put(3,4){\vector(-2,-3){1}}
\put(3,4){\vector(2,-3){1}}
\put(1.2,2){\framebox(1.6,.5){$i$ wins 1}}
\put(3.2,2){\framebox(1.6,.5){$i$ loses 1}}
\put(9,1.5){\vector(-2,-3){1}}
\put(9,1.5){\vector(2,-3){1}}
\put(7.2,-.5){\framebox(1.6,.5){$i$ wins 1}}
\put(9.2,-.5){\framebox(1.6,.5){$i$ loses 1}}
\put(1.4,3.2){heads}
\put(3.7,3.2){tails}
\put(7.4,.7){heads}
\put(9.7,.7){tails}
\end{picture}
\vglue5mm
\end{figure}
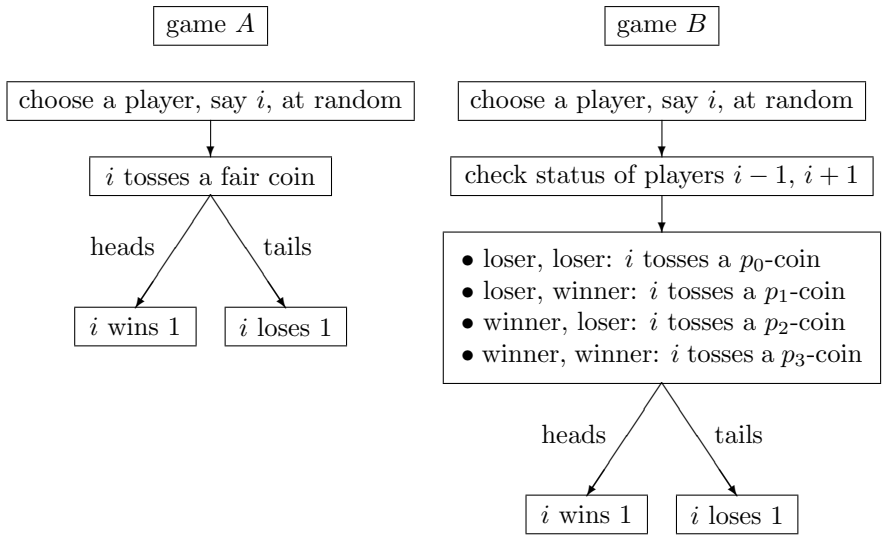

Toral used computer simulation to show that, with $N=50$, 100, or 200, $p_0=1$, $p_1=p_2=4/25$, and $p_3=7/10$, game $A$ is fair, game $B$ is losing, but both the random mixture $(1/2)(A+B)$ (toss a fair coin, play game $A$ if heads, game $B$ if tails) and the nonrandom periodic pattern $AABB$ (i.e., $AABBAABBAABB\cdots$), also denoted by $[2,2]$, are winning, providing new examples of \textit{Parrondo's paradox} (Harmer and Abbott \cite{HA02}, Abbott \cite{A10}).  Ethier and Lee \cite{EL12b} investigated $\mu_B$ and $\mu_{(1/2,1/2)}$, the mean profits per turn to the ensemble of $N$ players always playing game $B$ and always playing the randomly mixed game $(1/2)(A+B)$. These means were computable for $3\le N\le19$, at least, and appeared to converge as $N\to\infty$, suggesting that the Parrondo region (i.e., the region in which $\mu_B\le0$ and $\mu_{(1/2,1/2)}>0$) has nonzero volume in the limit.  In \cite{EL12c} this convergence was established under certain conditions on $p_0,p_1,p_2,p_3$, and the limits were expressed in terms of a parameterized spin system on the one-dimensional integer lattice.  In addition, assuming $p_1=p_2$, explicit formulas for the means were derived for $N=3,4,5,6$, and this allowed us to visualize the Parrondo region in these cases.  Incidentally, we note that the results can easily be extended to the more general random mixture $\gamma A+(1-\gamma)B$, in which case the mean profit per turn is denoted by $\mu_{(\gamma,1-\gamma)}$.

In this paper we replace the random mixture with the nonrandom periodic pattern $A^r B^s$, also denoted by $[r,s]$, where $r$ and $s$ are positive integers.  We show that $\mu_{[r,s]}$, the mean profit per turn to the ensemble of $N$ players repeatedly playing the pattern $A^r B^s$, is computable for $3\le N\le 18$ and $r+s\le4$, at least, and appears to converge as $N\to\infty$, albeit more slowly than in the random-mixture case.  Again this suggests that the Parrondo region ($\mu_B\le0$ and $\mu_{[r,s]}>0$) has nonzero volume in the limit.  Elsewhere \cite{EL12d} we establish this convergence under certain conditions on $p_0,p_1,p_2,p_3$ (see Appendix A herein for the conditions), and show that the limit as $N\to\infty$ coincides with $\lim_{N\to\infty}\mu_{(\gamma,1-\gamma)}$, where $\gamma:=r/(r+s)$.  In addition, assuming $p_1=p_2$, explicit formulas for $\mu_{[r,s]}$ can be derived for $N=3,4,5$ and $r+s\le3$, and this allows us to visualize the Parrondo region in these cases.

Nonrandom patterns of the form $A^rB^s$ have been part of the Parrondo paradox literature from its inception (Harmer and Abbott \cite{HA99}).  The reason is clear.  The flashing Brownian ratchet (Ajdari and Prost \cite{AP92}), which motivated the formulation of Parrondo's games, is most closely related to the pattern $AB$ (i.e., $ABABAB\cdots$).  But this pattern is not winning for the original capital-dependent games, so Parrondo considered $AABB$ instead.  Analysis of these patterns was initially based on simulation.  Analytical formulas for the mean profit per turn were found by Ekhad and Zeilberger \cite{EZ00} and Velleman and Wagon \cite{VW01}, and in the history-dependent setting by Kay and Johnson \cite{KJ03}.  Behrends \cite{B04} generalized Parrondo's games in such a way that the random-mixture case and the nonrandom-pattern case could be studied simultaneously, and he too found an expression for the mean profit per turn.  Ethier and Lee \cite{EL09} proved a strong law of large numbers and showed that the Parrondo effect appears (in the capital-dependent setting with the so-called bias parameter equal to 0) for the pattern $A^r B^s$ for all integers $r,s\ge1$ except $r=s=1$.  The problem of finding the most profitable pattern (not restricted to patterns of the form $A^r B^s$) was investigated in \cite{EZ00} and \cite{VW01} and eventually solved by Dinis \cite{D08}, the winning pattern being $ABABB$ (and its cyclic permutations) in the case of Parrondo's original parameters.  In Toral's \cite{T02} nonspatial $N$-player model with redistribution of wealth, Ethier and Lee \cite{EL12a} showed that the mean profit per turn for the pattern $A^r B^s$ converges, as $N\to\infty$, to the mean profit per turn for the random mixture $\gamma A+(1-\gamma)B$ with $\gamma:=r/(r+s)$, which does not depend on $N$.  This foreshadowed the results of the present paper and \cite{EL12d}.

\section{The Markov Chain and its Reduction}\label{MC}

The Markov chain formalized by Mihailovi\'c and Rajkovi\'c \cite{MR03} keeps track of the status (loser or winner, 0 or 1) of each of the $N\ge3$ players. Its state space is the product space
$$
\Sigma:=\{\bm x=(x_1,x_2,\ldots,x_N): x_i\in\{0,1\}{\rm\ for\ }i=1,\ldots,N\}=\{0,1\}^N
$$
with $2^N$ states.  Let $m_i(\bm x):=2x_{i-1}+x_{i+1}$; in other words, $m_i(\bm x)$ is the integer (0, 1, 2, or 3) whose binary representation is $(x_{i-1}\,x_{i+1})_2$;  of course $x_0:=x_N$ and $x_{N+1}:=x_1$.  Also, let $\bm x^i$ be the element of $\Sigma$ equal to $\bm x$ except at the $i$th component; for example, $\bm x^1:=(1-x_1,x_2,x_3,\ldots,x_N)$.

The one-step transition matrix $\bm P_B$ for this Markov chain depends not only on $N$ but on the four probability parameters $p_0$, $p_1$, $p_2$, and $p_3$, which we assume satisfy $0<p_m<1$ for $m=0,1,2,3$.  (This rules out Toral's choice of the probability parameters, at least for now, but we will return to this point in Sec.~\ref{reducible}.  We do not assume that $p_1=p_2$ until Sec.~\ref{region}.)  It has the form
\begin{equation}\label{x to x^i}
P_B(\bm x,\bm x^i):=\begin{cases}N^{-1}p_{m_i(\bm x)}&\text{if $x_i=0$,}\\N^{-1}q_{m_i(\bm x)}&\text{if $x_i=1$,}\end{cases}\qquad i=1,\ldots,N,\;\bm x\in\Sigma,
\end{equation}
\begin{equation}\label{x to x}
P_B(\bm x,\bm x):=N^{-1}\bigg(\sum_{i:x_i=0}q_{m_i(\bm x)}+\sum_{i:x_i=1}p_{m_i(\bm x)}\bigg),\qquad \bm x\in\Sigma,
\end{equation}
where $q_m:=1-p_m$ for $m=0,1,2,3$ and empty sums are 0, and $P_B(\bm x,\bm y)=0$ otherwise.

The description of the model suggests that its long-term behavior should be invariant under rotation (and, if $p_1=p_2$, reflection) of the $N$ players.  In order to maximize the value of $N$ for which exact computations are feasible, we will use the following lemma from \cite{EL12b} to make this precise.

\begin{lemma}\label{invariant}
Fix $N\ge3$, let $G$ be a subgroup of the symmetric group $S_N$.  For $\bm x:=(x_1,\ldots,x_N)\in\Sigma:=\{0,1\}^N$ and $\sigma\in G$, write $\bm x_\sigma:=(x_{\sigma(1)},\ldots,x_{\sigma(N)})\in\Sigma$.  Let $\bm P$ be the one-step transition matrix for a Markov chain in $\Sigma$ with a unique stationary distribution $\bm\pi$.  Assume that
\begin{equation}\label{G-invariance}
P(\bm x_\sigma,\bm y_\sigma)=P(\bm x,\bm y),\qquad \sigma\in G,\;\bm x,\bm y\in\Sigma.
\end{equation}
Then $\pi(\bm x_\sigma)=\pi(\bm x)$ for all $\sigma\in G$ and $\bm x\in\Sigma$.

Let us say that $\bm x\in\Sigma$ is equivalent to $\bm y\in\Sigma$ (written $\bm x\sim\bm y$) if there exists $\sigma\in G$ such that $\bm y=\bm x_\sigma$, and let us denote the equivalence class containing $\bm x$ by $[\bm x]$.  Then, in addition, $\bm P$ induces a one-step transition matrix $\bar{\bm P}$ for a Markov chain in the quotient set (i.e., the set of equivalence classes) $\Sigma/$$\sim$ defined by the formula
\begin{equation}\label{Pbar-def}
\bar{P}([\bm x],[\bm y]):=\sum_{\bm y':\bm y'\sim\bm y}P(\bm x,\bm y')=\sum_{\sigma\in G:\bm y_\sigma\;{\rm distinct}}P(\bm x,\bm y_\sigma),
\end{equation}
the second sum extending over only those $\sigma\in G$ for which the various $\bm y_\sigma$ are distinct.  Furthermore, if $\bar{\bm P}$ has a unique stationary distribution $\bar{\bm\pi}$, then the unique stationary distribution $\bm\pi$ is given by $\pi(\bm x)=\bar{\pi}([\bm x])/|[\bm x]|$, where $|[\bm x]|$ denotes the cardinality of the equivalence class $[\bm x]$.
\end{lemma}

As we saw in \cite{EL12b}, Lemma~\ref{invariant} applies to $\bm P_B$ of (\ref{x to x^i}) and (\ref{x to x}) with $G$ being the subgroup of cyclic permutations (or rotations) of $(1,2,\ldots,N)$.  If $p_1=p_2$, then Lemma~\ref{invariant} also applies to $\bm P_B$ with $G$ being the subgroup generated by the cyclic permutations and the order-reversing permutation (rotations and/or reflections) of $(1,2,\ldots,N)$, the dihedral group of order $2N$.  The resulting one-step transition matrix $\bar{\bm P}_B$ has smaller dimension than $\bm P_B$ (see Table~1 of \cite{EL12b}), thereby making matrix computations feasible for larger $N$.

If $p_0=p_1=p_2=p_3=1/2$, we denote $\bm P_B$ by $\bm P_A$.  Of course the same conclusions apply to $\bm P_A$.  We would like to argue that they also apply to $\bm P_A^r\bm P_B^s$ for each $r,s\ge1$.  For this we need to show that the mapping $\bm P\mapsto\bar{\bm P}$ has a certain multiplicative property.

Given a subgroup $G$ of $S_N$, let us say that a square matrix $\bm P$ (not necessarily stochastic) with rows and columns indexed by $\Sigma$ is \textit{$G$-invariant} if (\ref{G-invariance}) holds.  If $\bm P$ is $G$-invariant, then we can define $\bar{\bm P}$ as in (\ref{Pbar-def}).

\begin{lemma}\label{multiplication}
Let $G$ be a subgroup of $S_N$.  If the square matrices $\bm P_1$ and $\bm P_2$ are $G$-invariant, then $\bm P_1\bm P_2$ is $G$-invariant, and $\overline{\bm P_1\bm P_2}=\bar{\bm P}_1 \bar{\bm P}_2$.
\end{lemma}

\begin{proof}
The first assertion follows from the identity
\begin{eqnarray*}
[\bm P_1\bm P_2](\bm x_\sigma,\bm y_\sigma)&=&\sum_{\bm z\in\Sigma}P_1(\bm x_\sigma,\bm z)P_2(\bm z,\bm y_\sigma)=\sum_{\bm z\in\Sigma}P_1(\bm x_\sigma,\bm z_\sigma)P_2(\bm z_\sigma,\bm y_\sigma)\\
&=&\sum_{\bm z\in \Sigma}P_1(\bm x,\bm z)P_2(\bm z,\bm y)=[\bm P_1 \bm P_2](\bm x,\bm y).
\end{eqnarray*}
The second is a consequence of
\begin{eqnarray*}
[\overline{\bm P_1\bm P_2}]([\bm x],[\bm y])&=&\sum_{\bm y':\bm y'\sim\bm y}[\bm P_1\bm P_2](\bm x,\bm y')\\
&=&\sum_{\bm y':\bm y'\sim\bm y}\sum_{\bm z}P_1(\bm x,\bm z)P_2(\bm z,\bm y')\\
&=&\sum_{\bm z}P_1(\bm x,\bm z)\sum_{\bm y':\bm y'\sim\bm y}P_2(\bm z,\bm y')\\
&=&\sum_{\bm z}P_1(\bm x,\bm z)\bar{P}_2([\bm z],[\bm y])\\
&=&\sum_{[\bm z]}\sum_{\bm z':\bm z'\sim\bm z}P_1(\bm x,\bm z')\bar{P}_2([\bm z'],[\bm y])\\
&=&\sum_{[\bm z]}\bar{P}_1([\bm x],[\bm z])\bar{P}_2([\bm z],[\bm y])\\
&=&[\bar{\bm P}_1 \bar{\bm P}_2]([\bm x],[\bm y]),
\end{eqnarray*}
so the proof is complete.
\end{proof}

With $G$ being the subgroup of cyclic permutations or (if $p_1=p_2$) the dihedral group of order $2N$, $\bm P_A$ and $\bm P_B$ are of course $G$-invariant.  Consequently, for each $r,s\ge1$, Lemma~\ref{multiplication} tells us that $\bm P_A^r\bm P_B^s$ is $G$-invariant as well.  This fact will be needed in the next section.

\section{Strong Law of Large Numbers}
We will need the following version of the SLLN.

\begin{theorem}[Ethier and Lee \cite{EL09}]\label{slln}
Let $\bm P_A$ and $\bm P_B$ be one-step transition matrices for Markov chains in a finite state space $\Sigma_0$.  Fix $r,s\ge1$.  Assume that $\bm P:=\bm P_A^r\bm P_B^s$, as well as all cyclic permutations of $\bm P_A^r\bm P_B^s$, are irreducible and aperiodic, and let the row vector $\bm\pi$ be the unique stationary distribution of $\bm P$.  Given a real-valued function $w$ on $\Sigma_0\times\Sigma_0$,  define the payoff matrix $\bm W:=(w(i,j))_{i,j\in\Sigma_0}$.  Define $\dot{\bm P}_A:=\bm P_A\circ\bm W$ and $\dot{\bm P}_B:=\bm P_B\circ\bm W$, where $\circ$ denotes the Hadamard (entrywise) product, and put
$$
\mu_{[r,s]}:={1\over r+s}\bigg[\sum_{u=0}^{r-1}\bm\pi\bm P_A^u\dot{\bm P}_A\bm1+\sum_{v=0}^{s-1}\bm\pi\bm P_A^r\bm P_B^v\dot{\bm P}_B\bm1\bigg],
$$
where $\bm1$ denotes a column vector of $1$s with entries indexed by $\Sigma_0$.  Let $\{X_n\}_{n\ge0}$ be a nonhomogeneous Markov chain in $\Sigma_0$ with one-step transition matrices $\bm P_A,\ldots,\bm P_A$ $(r\text{ times})$, $\bm P_B,\ldots,\bm P_B$ $(s\text{ times})$, $\bm P_A,\ldots,\bm P_A$ $(r\text{ times})$, $\bm P_B,\ldots,\bm P_B$ $(s\text{ times})$, and so on, and let the initial distribution be arbitrary.  For each $n\ge1$, define $\xi_n:=w(X_{n-1},X_n)$ and $S_n:=\xi_1+\cdots+\xi_n$.  Then $\lim_{n\to\infty}n^{-1}S_n=\mu_{[r,s]}$ \emph{a.s.}
\end{theorem}

Here ``a.s.'' stands for ``almost surely,'' meaning ``with probability 1.''

Our original Markov chain has state space $\Sigma:=\{0,1\}^N$ and its one-step transition matrix $\bm P_B$ is given by (\ref{x to x^i}) and (\ref{x to x}), but Theorem~\ref{slln} does not apply directly.  Instead, we augment the state space, letting $\Sigma^*:=\Sigma\times \{1,2,\ldots,N\}$ and keeping track not only of the status of each player as described by $\bm x\in\Sigma$ but also of the label of the next player to play, say $i$.  The new one-step transition matrix $\bm P_B^*$ has the form
$$
P_B^*((\bm x,i),(\bm x^i,j)):=\begin{cases}N^{-1}p_{m_i(\bm x)}&\text{if $x_i=0$,}\\N^{-1}q_{m_i(\bm x)}&\text{if $x_i=1$,}\end{cases}\qquad (\bm x,i)\in\Sigma^*,\; j=1,2,\ldots,N,
$$
$$
P_B^*((\bm x,i),(\bm x,j)):=\begin{cases}N^{-1}q_{m_i(\bm x)}&\text{if $x_i=0$,}\\N^{-1}p_{m_i(\bm x)}&\text{if $x_i=1$,}\end{cases}\qquad (\bm x,i)\in\Sigma^*,\; j=1,2,\ldots,N,
$$
where $q_m:=1-p_m$ for $m=0,1,2,3$, and $P_B^*((\bm x,i),(\bm y,j))=0$ otherwise.  If $p_0=p_1=p_2=p_3=1/2$, we denote $\bm P_B^*$ by $\bm P_A^*$.  The one-step transition matrix $\bm P^*:=(\bm P_A^*)^r(\bm P_B^*)^s$, as well as all cyclic permutations of $(\bm P_A^*)^r(\bm P_B^*)^s$, are irreducible and aperiodic, so let $\bm\pi^*$ be the unique stationary distribution of $\bm P^*$.  The payoff matrix $\bm W$ now has each nonzero entry equal to $\pm1$, so Theorem~\ref{slln} applies.  Since $\dot{\bm P}_A^*\bm1=\bm0$, we have
\begin{equation}\label{mu via P*}
\mu_{[r,s]}={1\over r+s}\sum_{v=0}^{s-1}\bm\pi^*(\bm P_A^*)^r(\bm P_B^*)^v\dot{\bm P}_B^*\bm1,
\end{equation}
where $\bm1$ denotes a column vector of $1$s with entries indexed by $\Sigma^*$.
Here $\dot{\bm P}_B^*$ is the Hadamard (entrywise) product ${\bm P}_B^*\circ {\bm W}$, where ${\bm W}$ is the matrix each of whose entries is $1$, $-1$, or $0$ according to whether the corresponding entry of ${\bm P}_B^*$ is of the form $N^{-1}p_m$, $N^{-1}q_m$, or $0$.

\begin{lemma}\label{pistar}
Fix $r,s\ge1$ and let $\bm\pi$ be the unique stationary distribution of $\bm P^r_A \bm P^s_B$.  Then $(\bm P_A^*)^r(\bm P_B^*)^s$ has unique stationary distribution $\pi^*(\bm x,i):=\pi(\bm x)N^{-1}$ and
\begin{equation}\label{indep}
[\bm\pi^*(\bm P_A^*)^r(\bm P_B^*)^v](\bm x,i)=[\bm\pi\bm P_A^r\bm P_B^v](\bm x)N^{-1}
\end{equation}
for $v=0,1,\ldots,s-1$.
\end{lemma}

\begin{proof}
Let $\bm\pi^*$ be stationary for $(\bm P_A^*)^r(\bm P_B^*)^s$.  We will show that it has the form stated in the lemma.  Let $\bm X^*(0),\ldots,\bm X^*(r+s)$ be a nonhomogeneous Markov chain in $\Sigma^*$ with one-step transition matrices $\bm P_A^*,\ldots,\bm P_A^*$ ($r$ times), $\bm P_B^*,\ldots,\bm P_B^*$ ($s$ times), and initial distribution $\bm\pi^*$.  Then $\bm X^*(r+s)$ has distribution $\bm\pi^*(\bm P_A^*)^r(\bm P_B^*)^s=\bm\pi^*$.  Let $\bm X(0),\ldots,\bm X(r+s)$ be the projections of $\bm X^*(0),\ldots,\bm X^*(r+s)$ onto $\Sigma$, and let $\bm\pi$ be the marginal of $\bm\pi^*$ on $\Sigma$.  Then $\bm X(0),\ldots,\bm X(r+s)$ is a nonhomogeneous Markov chain in $\Sigma$ with one-step transition matrices $\bm P_A,\ldots,\bm P_A$ ($r$ times), $\bm P_B,\ldots,\bm P_B$ ($s$ times), and initial distribution $\bm\pi$.  Furthermore, $\bm X(r+s)$ has distribution $\bm\pi\bm P_A^r\bm P_B^s$ as well as distribution $\bm\pi$.  So $\bm\pi$ is the unique stationary distribution of $\bm P^r_A \bm P^s_B$, as assumed in the statement of the lemma.  On the other hand, writing $\bm X^*(r+s)=(\bm X(r+s),I(r+s))$, $I(r+s)$ is independent of $\bm X(r+s)$ and uniform$\{1,2,\ldots,N\}$.  So the distribution $\bm\pi^*$ of $\bm X^*(r+s)$ must have the stated form.

Finally, fix $v\in\{0,1,\ldots,s-1\}$ and write $\bm X^*(r+v)=(\bm X(r+v),I(r+v))$.  The left side of (\ref{indep}) is the distribution of $\bm X^*(r+v)$, while $[\bm\pi\bm P_A^r\bm P_B^v](\bm x)$ is the distribution of $\bm X(r+v)$.  Further, $I(r+v)$ is independent of $\bm X(r+v)$ and uniform$\{1,2,\ldots,N\}$ by virtue of the one-step transition matrices $\bm P_A^*$ and $\bm P_B^*$.  This implies (\ref{indep}).
\end{proof}

It follows from (\ref{mu via P*}) and Lemma~\ref{pistar} that
\begin{equation}\label{mu via P}
\mu_{[r,s]}={1\over r+s}\sum_{v=0}^{s-1}\bm\pi\bm P^r_A \bm P^v_B\dot{\bm P}_B\bm1,
\end{equation}
where $\dot{\bm P}_B$ is obtained from ${\bm P}_B$ by replacing each $q_m$ by $-q_m$ (prior to any simplification using $q_m=1-p_m$) and  $\bm1$ denotes a column vector of $1$s with entries indexed by $\Sigma$.
This expresses the mean profit $\mu_{[r,s]}$ in terms of the original one-step transition matrices $\bm P_A$ and $\bm P_B$ and the unique stationary distribution $\bm\pi$ of $\bm P_A^r\bm P_B^s$.  However, for computational purposes, we would like to express it in terms of the reduced one-step transition matrices $\bar{\bm P}_A$ and $\bar{\bm P}_B$ and the unique stationary distribution $\bar{\bm\pi}$ of $\bar{\bm P}_A^r\bar{\bm P}_B^s$.  The subgroup $G$ implicit in the notation is either the subgroup of cyclic permutations of $(1,2,\ldots,N)$ or (if $p_1=p_2$) the dihedral group of order $2N$.

\begin{theorem}\label{pattern_mean}
Given $r,s\ge1$, let $\{(\bm X(n), I(n))\}_{n\ge0}$ be the nonhomogeneous Markov chain in $\Sigma^*$ with one-step transition matrices $\bm P_A^*,\ldots,\bm P_A^*$ $(r\text{ times})$, $\bm P_B^*,\ldots,\bm P_B^*$ $(s\text{ times})$, $\bm P_A^*,\ldots,\bm P_A^*$ $(r\text{ times})$, $\bm P_B^*,\ldots,\bm P_B^*$ $(s\text{ times})$, and so on, and arbitrary initial distribution.  Define
$$
\xi_n:=w((\bm X(n-1),I(n-1)),(\bm X(n),I(n))), \qquad n\ge1,
$$
where the payoff function $w$ is $1$ for a win and $-1$ for a loss, determined by whether the corresponding entry of $\bm P_B^*$ is of the form $N^{-1}p_m$ or $N^{-1}q_m$.  Let $S_n:=\xi_1+\cdots+\xi_n$ for each $n\ge1$.  Then $n^{-1}S_n\to\mu_{[r,s]}$ \emph{a.s.}, where
\begin{equation}\label{mu via Pbar}
\mu_{[r,s]}={1\over r+s}\sum_{v=0}^{s-1}\bar{\bm\pi}\bar{\bm P}_A^r \bar{\bm P}_B^v \dot{\bar{\bm P}}_B\bm1,
\end{equation}
$\dot{\bar{\bm P}}_B$ is obtained from ${\bar{\bm P}}_B$ by replacing each $q_m$ by $-q_m$ (prior to any simplification using $q_m=1-p_m$), and $\bm 1$ denotes a column vector of $1$s with entries indexed by $\Sigma/$$\sim$.
\end{theorem}

\begin{proof}
This is an application of Theorem~\ref{slln}.
To show that the right sides of (\ref{mu via P}) and (\ref{mu via Pbar}) are equal, we note that, for any $G$-invariant square matrix $\bm Q$, we have, using Lemma~\ref{invariant},
\begin{eqnarray}\label{Q to Qbar}
\bm\pi{\bm Q}\bm1&=&\sum_{\bm x\in\Sigma}\;\sum_{\bm y\in\Sigma}\pi(\bm x){Q}(\bm x,\bm y)\nonumber\\
&=&\sum_{[\bm x]\in\Sigma/\sim}\;\sum_{\bm x':\bm x'\sim\bm x}\;\sum_{[\bm y]\in\Sigma/\sim}\;\sum_{\bm y':\bm y'\sim\bm y}\pi(\bm x'){Q}(\bm x',\bm y')\nonumber\\
&=&\sum_{[\bm x]\in\Sigma/\sim}\;\sum_{\bm x':\bm x'\sim\bm x}\;\sum_{[\bm y]\in\Sigma/\sim}\pi(\bm x')\bar{Q}([\bm x'],[\bm y])\nonumber\\
&=&\sum_{[\bm x]\in\Sigma/\sim}\;\sum_{[\bm y]\in\Sigma/\sim}\pi(\bm x)|[\bm x]|{\bar{Q}}([\bm x],[\bm y])\nonumber\\
&=&\sum_{[\bm x]\in\Sigma/\sim}\;\sum_{[\bm y]\in\Sigma/\sim}\bar{\pi}([\bm x]){\bar{Q}}([\bm x],[\bm y])\nonumber\\
&=&\bar{\bm\pi}{\bar{\bm Q}}{\bm1},
\end{eqnarray}
where $\bm 1$ denotes a column vector of $1$s of the appropriate dimension.
Now if $\bm P_B$ is $G$-invariant, then so too is $\dot{\bm P}_B$, and likewise for
$\bm Q_v:={\bm P}_A^r{\bm P}_B^v\dot{{\bm P}}_B$ by Lemma~\ref{multiplication};  using (\ref{Q to Qbar}) and Lemma~\ref{multiplication},
$$
\bm\pi{\bm P}_A^r{\bm P}_B^v\dot{{\bm P}}_B\bm1=\bm\pi\bm Q_v\bm1=\bar{\bm\pi}\bar{\bm Q}_v\bm1=\bar{\bm\pi} \bar{\bm P}_A^r \bar{\bm P}_B^v \dot{\bar{{\bm P}}}_B\bm1
$$
for $v=0,1,\ldots,s-1$.  (Note that $\bar{\dot{\bm P}}=\dot{\bar{\bm P}}$.)
\end{proof}

We conclude with an application of the SLLN for this model.

Let us denote $\mu_B$, the mean profit per turn to the ensemble of $N$ players always playing game $B$, by $\mu_B(p_0,p_1,p_2,p_3)$ to emphasize its dependence on the parameter vector.  As shown in \cite{EL12b},
\begin{equation}\label{couple B}
\mu_B(p_0,p_1,p_2,p_3)=-\mu_B(q_3,q_2,q_1,q_0),
\end{equation}
where $q_m:=1-p_m$ for $m=0,1,2,3$.
Fix $r,s\ge1$.  Let us denote $\mu_{[r,s]}$ of (\ref{mu via P*}), (\ref{mu via P}), or (\ref{mu via Pbar}) by $\mu_{[r,s]}(p_0,p_1,p_2,p_3)$.

\begin{corollary}\label{coupling}
$$
\mu_{[r,s]}(p_0,p_1,p_2,p_3)=-\mu_{[r,s]}(q_3,q_2,q_1,q_0).
$$
In particular, if $p_0+p_3=1$ and $p_1+p_2=1$, then $\mu_{[r,s]}(p_0,p_1,p_2,p_3)=0$.
\end{corollary}

\begin{proof}
(\ref{couple B}) was proved in \cite{EL12b} using a coupling argument.  The proof of the corollary is the same, except that the (homogeneous) Markov chain is replaced by the nonhomogeneous one in the statement of Theorem~\ref{pattern_mean}.
\end{proof}

\section{Reducible Cases}\label{reducible}

We have assumed that $0<p_m<1$ for $m=0,1,2,3$, which ensures that the Markov chain for game $B$ is irreducible and aperiodic. This assumption was weakened to some extent in \cite{EL12b}.  We can derive the results obtained here under analogous conditions by noticing that the irreducibility and aperiodicity assumptions of Theorem~\ref{slln} can be weakened to simply ergodicity, which requires a unique stationary distribution and convergence of the distribution at time $n$ to that stationary distribution as $n\to\infty$, regardless of the initial distribution.

Fix $r,s\ge1$. We first notice that the Markov chain in $\Sigma$ with one-step transition matrix $\bm {\bm P}_A^{r-1}{\bm P}_B^s\bm P_A$, $\bm {\bm P}_A^{r-2}{\bm P}_B^s\bm P_A^2$, \dots, or ${\bm P}_B^s{\bm P}_A^r$ is ergodic (in fact, irreducible and aperiodic) for all choices of the parameters $p_0,p_1,p_2,p_3\in[0,1]$.  It remains to check the ergodicity of the Markov chain in $\Sigma$ with one-step transition matrix ${\bm P}_B^{s-1}{\bm P}_A^r\bm P_B$, ${\bm P}_B^{s-2}{\bm P}_A^r\bm P_B^2$, \dots, or ${\bm P}_A^r{\bm P}_B^s$.  We denote by $\bm0\in\Sigma$ the state consisting of all 0s, and by $\bm1\in\Sigma$ the state consisting of all $1$s.

\begin{enumerate}
\item Suppose $p_0=1$ and $0<p_m<1$ for $m=1,2,3$, as Toral \cite{T01} originally assumed.  Then state $\bm 0$ cannot be reached from $\Sigma-\{\bm 0\}$ by $\bm P_B$ and therefore by ${\bm P}_B^{s-1}{\bm P}_A^r\bm P_B$, ${\bm P}_B^{s-2}{\bm P}_A^r\bm P_B^2$, \dots, and ${\bm P}_A^r{\bm P}_B^s$.  For these one-step transition matrices, $\Sigma-\{\bm0\}$ is irreducible and aperiodic while $\bm0$ is transient.  Thus, the required ergodicity holds.  Theorem~\ref{slln} (modified as noted above) and Theorem~\ref{pattern_mean} apply.

\item Suppose $p_0=0$ and $0<p_m<1$ for $m=1,2,3$.  Then state $\bm 0$ is absorbing for $\bm P_B$, and absorption occurs with probability 1, so $\mu_B=-1$. However ${\bm P}_B^{s-1}{\bm P}_A^r\bm P_B$, ${\bm P}_B^{s-2}{\bm P}_A^r\bm P_B^2$, \dots, and ${\bm P}_A^r{\bm P}_B^s$ are all irreducible and aperiodic in $\Sigma$. Consequently, Theorem~\ref{slln} (as stated) and Theorem~\ref{pattern_mean} apply.

\item Suppose $p_3=0$ and $0<p_m<1$ for $m=0,1,2$.  This is analogous to case 1 but with the role of $\bm0$ played by $\bm1$.

\item Suppose $p_3=1$ and $0<p_m<1$ for $m=0,1,2$.  This is analogous to case 2, except that state $\bm 1$ is absorbing for game $B$ and $\mu_B=1$.

\item Suppose $p_0=1$, $p_3=0$, and $0<p_m<1$ for $m=1,2$.  Then states $\bm 0$ and $\bm 1$ cannot be reached from $\Sigma-\{\bm 0,\bm 1\}$ by $\bm P_B$ and therefore by ${\bm P}_B^{s-1}{\bm P}_A^r\bm P_B$, ${\bm P}_B^{s-2}{\bm P}_A^r\bm P_B^2$, \dots, and ${\bm P}_A^r{\bm P}_B^s$. For these one-step transition matrices, $\Sigma-\{\bm0,\bm1\}$ is irreducible and aperiodic while $\bm0$ and $\bm1$ are transient.  Thus, the required ergodicity holds.  Again, Theorem~\ref{slln} (modified) and Theorem~\ref{pattern_mean} apply.

\item Suppose $p_0=0$, $p_3=1$, and $0<p_m<1$ for $m=1,2$.  Then both $\bm 0$ and $\bm 1$ are absorbing for $\bm P_B$, and absorption occurs with probability 1. However, the Markov chain with one-step transition matrix ${\bm P}_B^{s-1}{\bm P}_A^r\bm P_B$, ${\bm P}_B^{s-2}{\bm P}_A^r\bm P_B^2$, \dots, or ${\bm P}_A^r{\bm P}_B^s$ has different behavior.  If $N$ is odd, it is irreducible and aperiodic in $\Sigma$, whereas if $N$ is even, $\Sigma-\{01\cdots01, 10\cdots10\}$ is irreducible and aperiodic and $01\cdots01$ and $10\cdots10$ are transient.  Again, Theorem~\ref{slln} (as stated if $N$ is odd, modified if $N$ is even) and Theorem~\ref{pattern_mean} apply.
\end{enumerate}

\section{The Parrondo Region}
\label{region}
In this section we assume that $p_1=p_2$, so that the coin tossed in game $B$ depends only on the number of winners among the two nearest neighbors. For fixed $N\ge3$ and $r, s\ge1$, we have three free probability parameters, $p_0, p_1, $ and $p_3$, so our parameter space is the unit cube $(0,1)^3 :=(0,1)\times(0,1)\times(0,1)$. With caution (see Sec.~\ref{reducible}), we can also include parts of the boundary.  Of interest are the \textit{Parrondo region}, the subset of the parameter space in which the Parrondo effect (i.e., $\mu_B\le0$ and $\mu_{[r,s]}>0$) appears, and the \textit{anti-Parrondo region}, the subset of the parameter space in which the anti-Parrondo effect (i.e., $\mu_B\ge0$ and $\mu_{[r,s]}<0$) appears.

\begin{theorem}\label{volume}
Fix $N\ge3$ and $r,s\ge1$.  With $q_m:=1-p_m$ for $m=0,1,3$, the parameter vector $(p_0,p_1,p_3)$ belongs to the Parrondo region if and only if the parameter vector $(q_3,q_1,q_0)$ belongs to the anti-Parrondo region.  In particular, the Parrondo region and the anti-Parrondo region have the same (three-dimensional) volume.
\end{theorem}

\begin{proof}
Let $\mu_B$ and $\mu_{[r,s]}$ denote the means for the parameter vector $(p_0,p_1,p_3)$, and let $\mu_B^*$ and $\mu_{[r,s]}^*$ denote the means for the parameter vector $(q_3,q_1,q_0)$.  Then, by (\ref{couple B}) and Corollary~\ref{coupling}, $\mu_B=\mu_B(p_0,p_1,p_1,p_3)=-\mu_B(q_3,q_1,q_1,q_0)=-\mu_B^*$ and $\mu_{[r,s]}=\mu_{[r,s]}(p_0,p_1,p_1,p_3)=-\mu_{[r,s]}(q_3,q_1,q_1,q_0)=-\mu_{[r,s]}^*$.  Therefore,
$\mu_B\le0$ and $\mu_{[r,s]}>0$ if and only if $\mu_B^*\ge0$ and $\mu_{[r,s]}^*<0$.  Since the transformation
$$
\Lambda(p_0,p_1,p_3):=(1-p_3,1-p_1,1-p_0)
$$
from $(0,1)^3$ to $(0,1)^3$ has Jacobian identically equal to $1$, the final conclusion holds.
\end{proof}

We can therefore focus our attention in what follows on the Parrondo region.

Using the mean formula (\ref{mu via Pbar}) in terms of the reduced Markov chains, together with the assumption that $p_1=p_2$, when $N=3$ we find that
\begin{eqnarray*}
\mu_{[1,1]}&=&{5 [2 p_1 (3 + p_0 + q_3) - 3 q_0 - 5 q_3]\over
  2 (17 + 15 p_0 + 4 p_0 p_1 + 8 p_0 q_3 + 4 p_1 p_3 + 4 q_1 + 19 q_3)},\\
\mu_{[1,2]}&=&2 [-494 + 287 p_0 - 51 p_0^2 + 181 p_3 - 13 p_0 p_3 - 12 p_0^2 p_3 +142 p_3^2 -40 p_0 p_3^2\\
 &&\quad{} + (520 + 61 p_0 - 65 p_0^2 - 113 p_3 + 130 p_0 p_3 - 28 p_0^2 p_3 - 65 p_3^2 + 28 p_0 p_3^2) p_1 \\
 &&\quad{} - 2 (1 - p_0 - p_3) (13 + 7 p_0 - 7 p_3) p_1^2]/\{3 [494 + 335 p_0 - 154 p_0^2 \\
 &&\quad{} - 157 p_3 - 64 p_0 p_3 + 32 p_0^2 p_3 - 130 p_3^2 - 64 p_0 p_3^2 + 32 p_0^2 p_3^2  \\
 &&\quad{} - 2 (1 - p_0 - p_3) (89 - 8 p_0 + 16 p_3 - 16 p_0 p_3) p_1 + 8 (1 - p_0 - p_3)^2 p_1^2]\}, \\
\mu_{[2,1]}&=&{38 [p_1 (12 + p_0 + q_3) - 6 q_0 - 7 q_3]\over
  3 (367 + 111 p_0 + 8 p_0 p_1 + 16 p_0 q_3 + 8 p_1 p_3 + 8 q_1 + 119 q_3)}.
\end{eqnarray*}
Since
\begin{eqnarray*}
\mu_B={p_1 (p_0 + q_3) - q_3\over p_0 p_1 + 2 p_0 q_3 + q_1 q_3}
\end{eqnarray*}
(see \cite{EL12b}), the Parrondo region for the $[1,1]$ pattern is described by
$p_1 (p_0 + q_3) - q_3\le 0$ and $2 p_1 (3 + p_0 + q_3) - 3 q_0 - 5 q_3>0$,
or, equivalently,
\begin{equation*}
{3q_0+5q_3\over2(3+p_0+q_3)}<p_1\le{q_3\over p_0+q_3}.
\end{equation*}
Similarly, the Parrondo region for the $[2,1]$ pattern is described by
\begin{equation*}
{6 q_0 + 7 q_3\over12 + p_0 + q_3}<p_1\le{q_3\over p_0+q_3}.
\end{equation*}
With the parameter space being the $(p_0,p_3,p_1)$ unit cube, the Parrondo region is the union of two connected components.  See row 1 of Figure~\ref{region_fig}.  For the $[1,1]$ pattern, integration yields its exact volume, $(100\ln5+36\ln3-276\ln2-9)/8\approx0.0231515$.   For the $[1,2]$ pattern, numerical integration gives its approximate volume, 0.0166398.  For the $[2,1]$ pattern, integration results in its exact volume (rounded), 0.0268219.

In \cite{EL12b} we pointed out a close relationship between the spatially dependent Parrondo games with $N=3$ and the history-dependent Parrondo games of Parrondo, Harmer, and Abbott \cite{PHA00}.  Specifically, the Parrondo regions for the equally weighted random mixtures were shown to be equal in the two models.  The analogous result for the nonrandom pattern $[1,1]$ fails, however.

When $N=4$ or $N=5$, we have explicit, albeit very complicated, formulas for $\mu_{[r,s]}$ with $r+s\le 3$. It would be impractical to state those formulas here. Instead, we plot the Parrondo regions and the anti-Parrondo regions for these cases in rows 2 and 3 of Figure~\ref{region_fig}.  They are surprisingly similar to the regions for the random-mixture cases depicted in \cite{EL12b}.

\begin{figure}[ht]
\centering
\includegraphics[width = 1.5in]{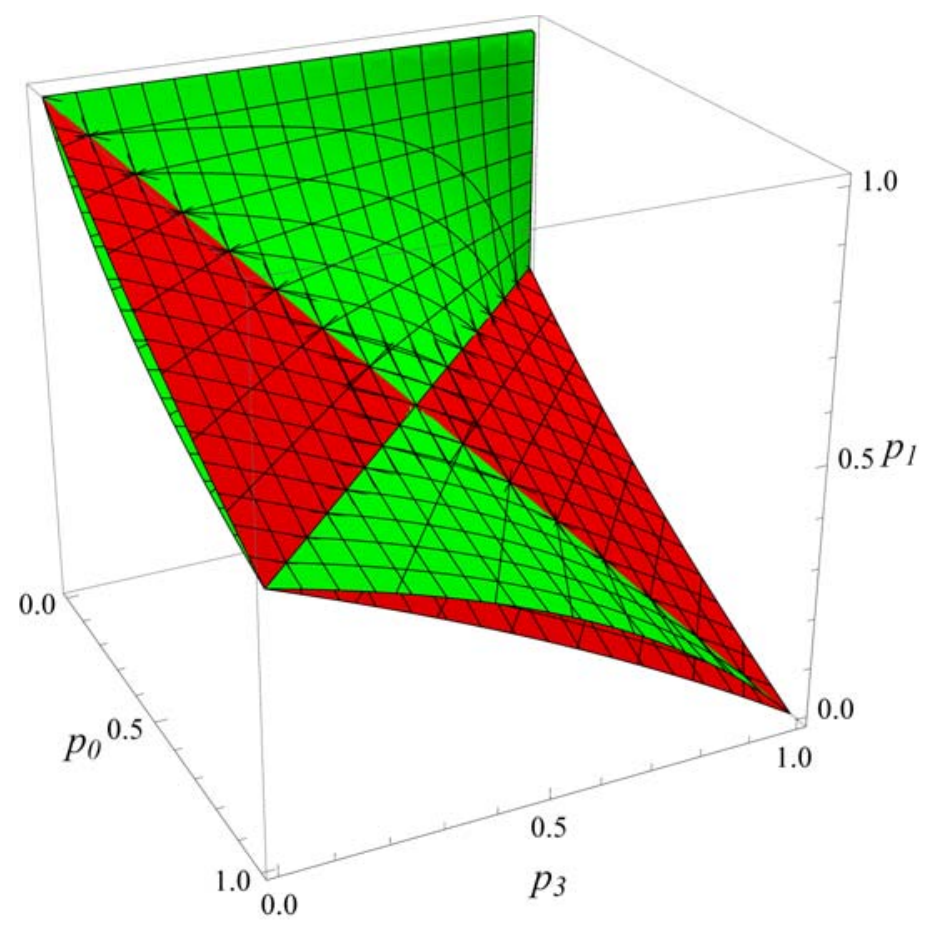}\;
\includegraphics[width = 1.5in]{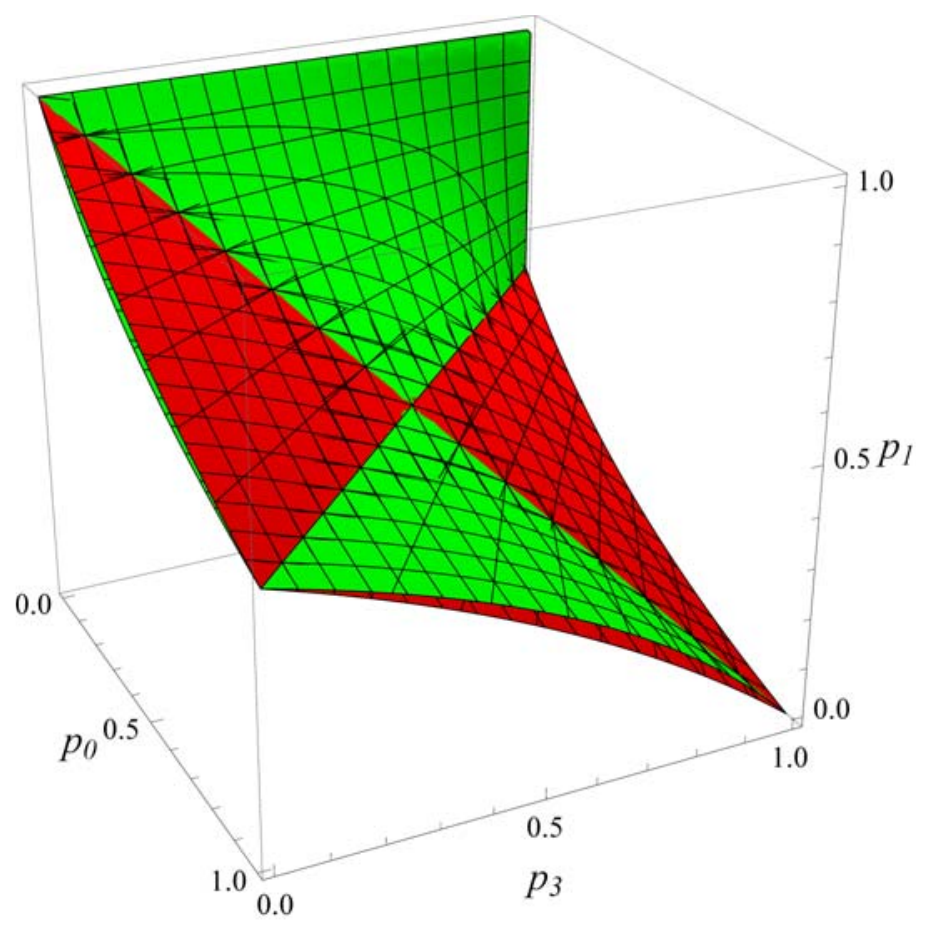}\;
\includegraphics[width = 1.5in]{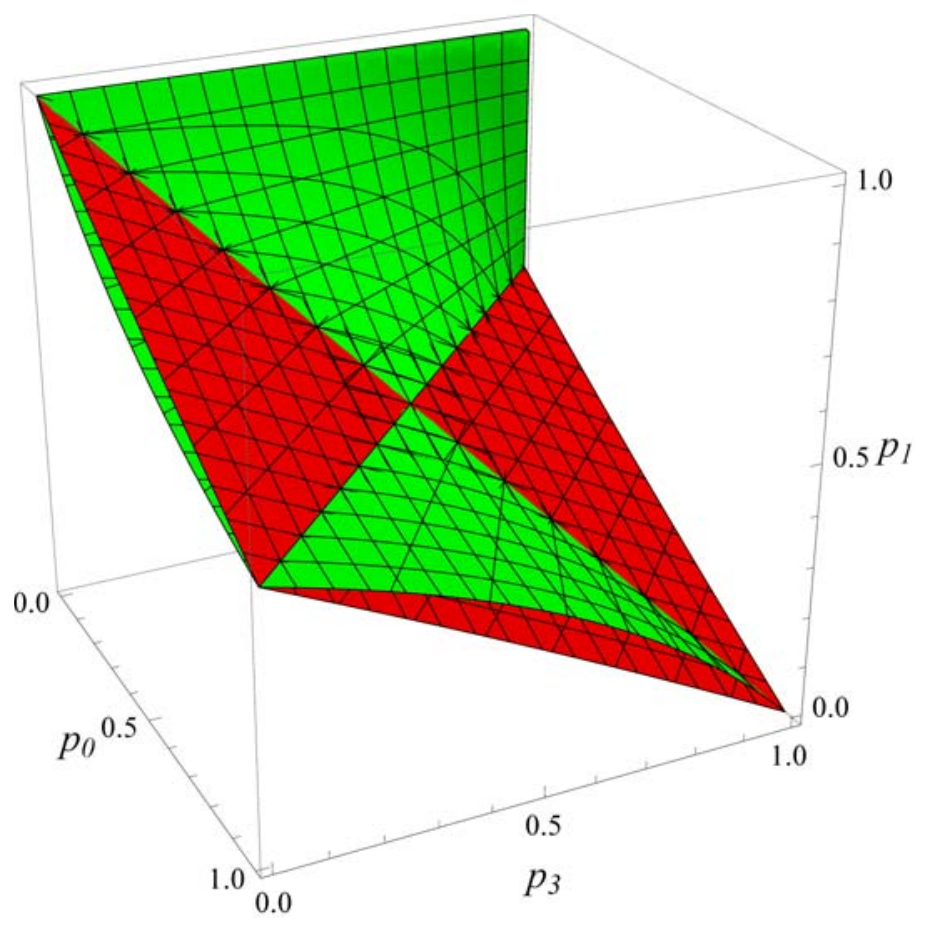}

\medskip
\includegraphics[width = 1.5in]{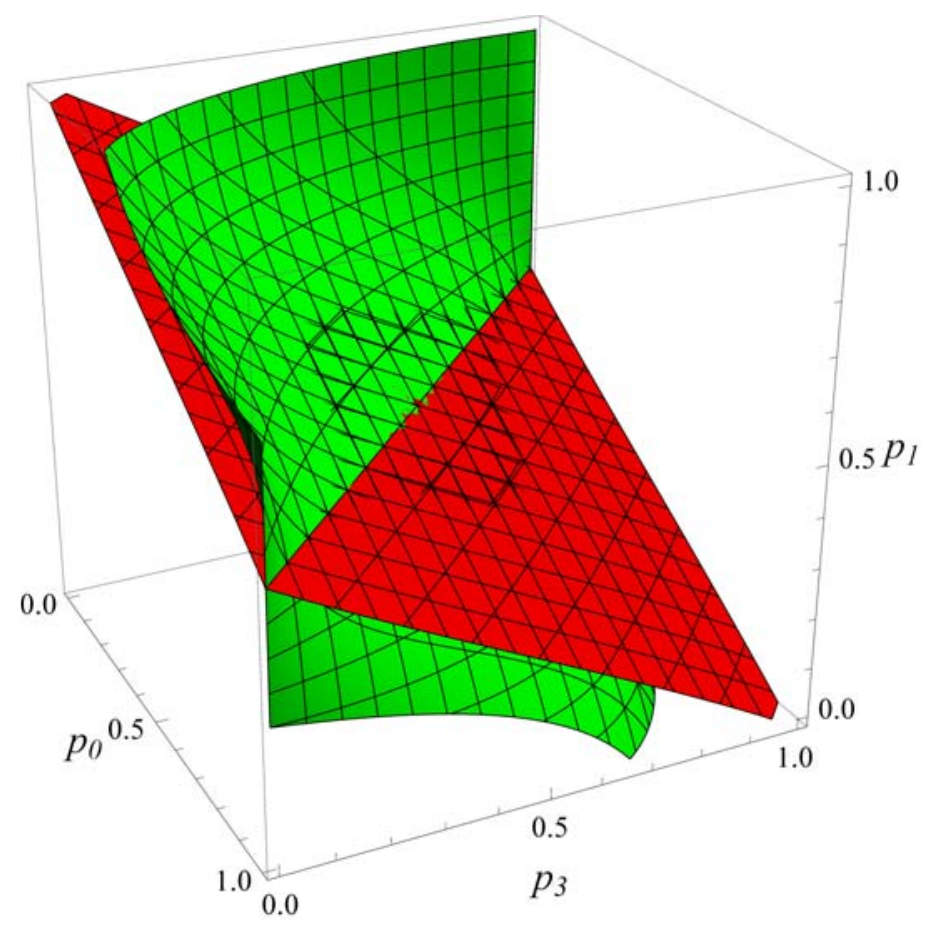}\;
\includegraphics[width = 1.5in]{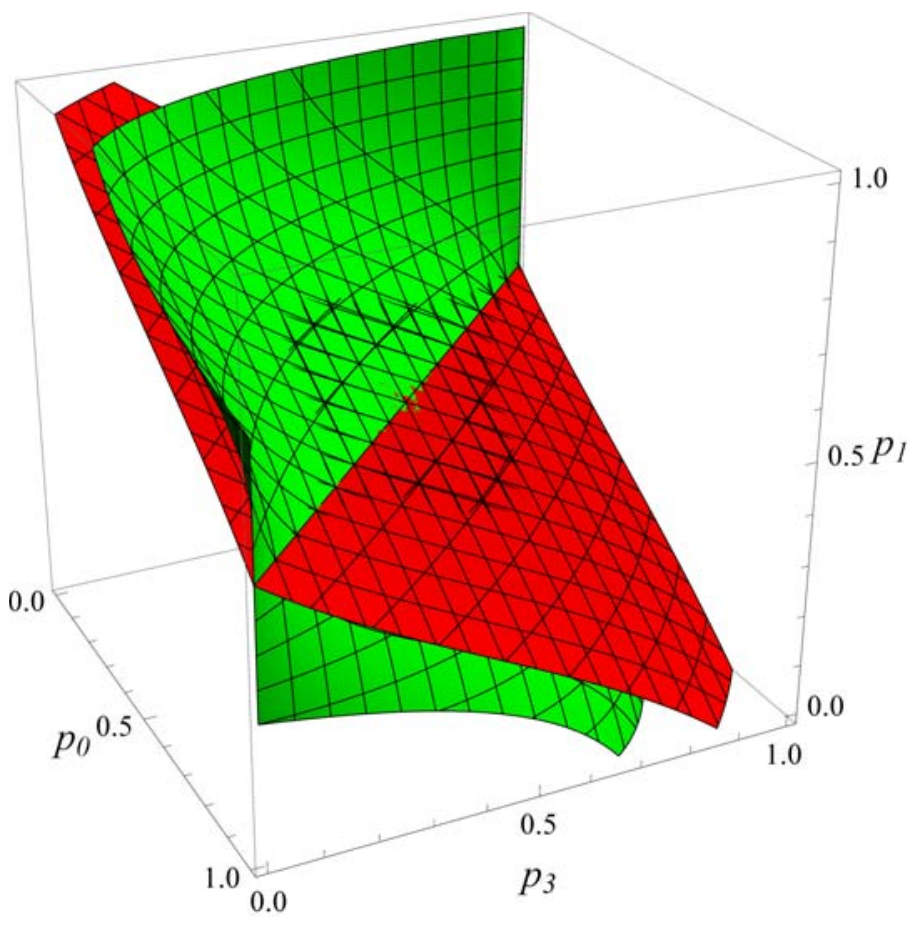}\;
\includegraphics[width = 1.5in]{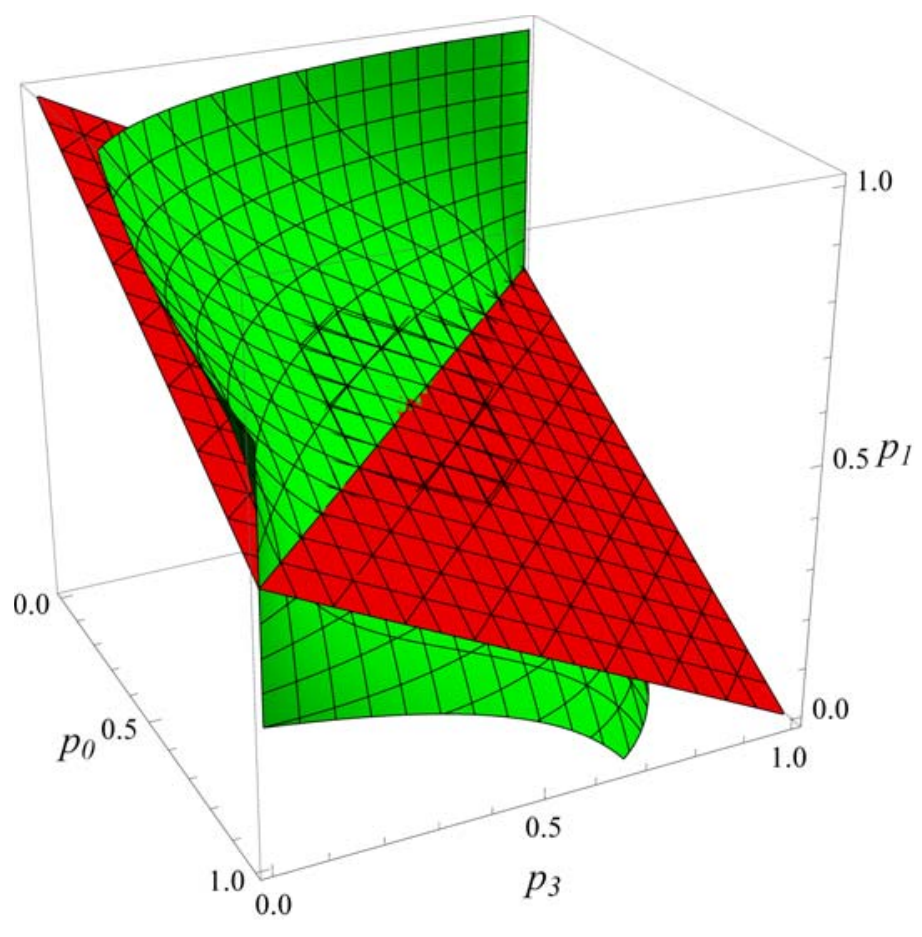}

\medskip
\includegraphics[width = 1.5in]{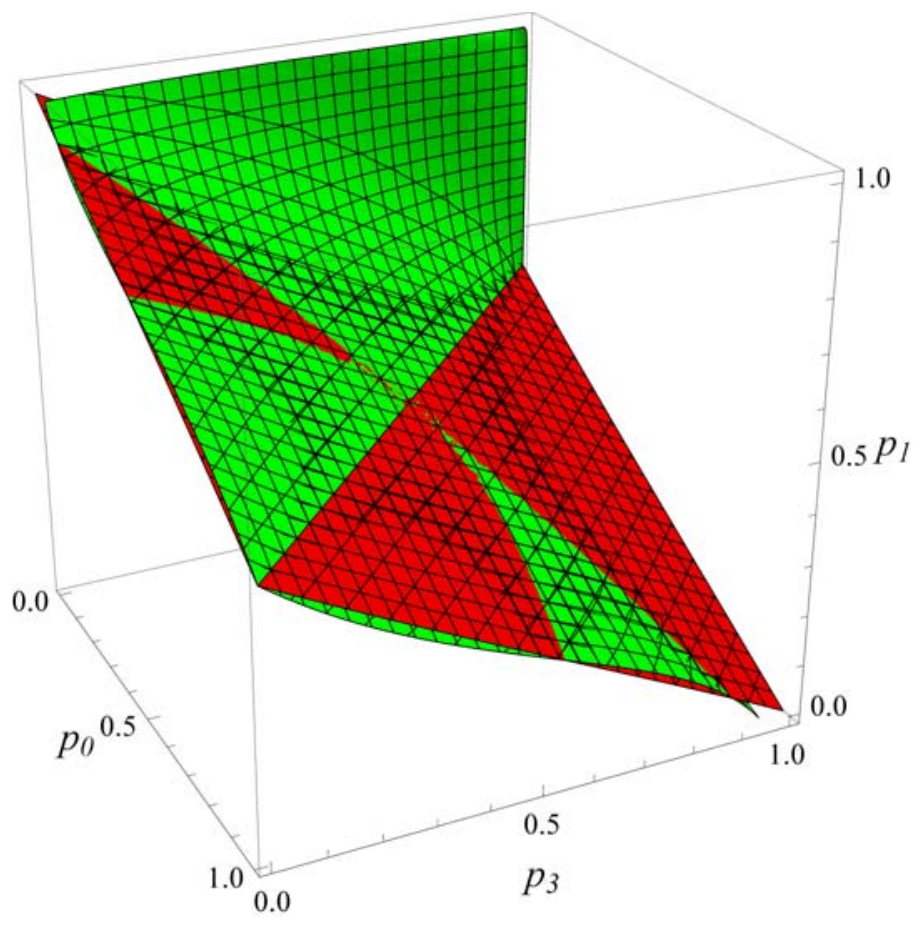}\;
\includegraphics[width = 1.5in]{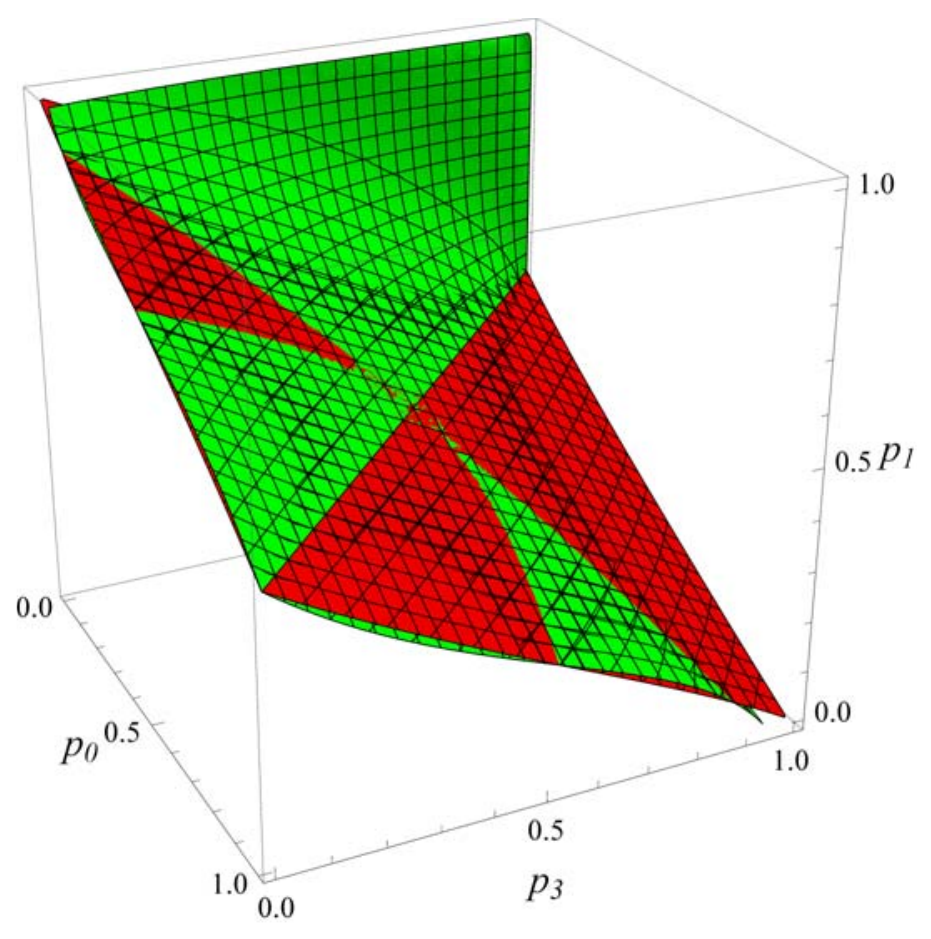}\;
\includegraphics[width = 1.5in]{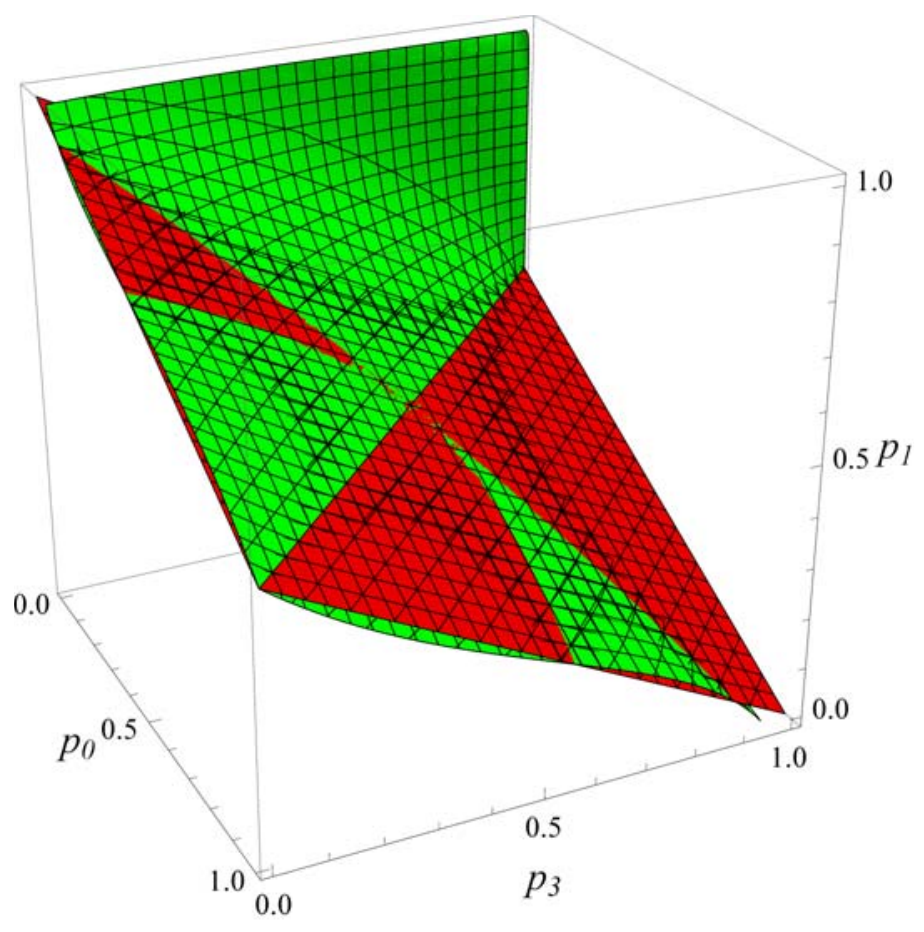}
\caption{\label{region_fig}When $N=3$ (row 1), $N=4$ (row 2), or $N=5$ (row 3), the green (or light) surface is the surface $\mu_B=0$, and the red (or dark) surface is the surface $\mu_{[r,s]}=0$ with $[r,s]=[1,1]$ (column 1), $[r,s]=[1,2]$ (column 2), or $[r,s]=[2,1]$ (column 3), all in the $(p_0,p_3,p_1)$ unit cube.  The Parrondo region is the region on or below the green surface and above the red surface, while the anti-Parrondo region is the region on or above the green surface and below the red surface.  (Assumption: $p_1=p_2$.)}
\end{figure}
\afterpage{\clearpage}

Table~\ref{ex1} displays the mean profits $\mu_{[r,s]}$ for $3\le N \le 18$ and the six choices of $[r,s]$ with $r+s\le 4$ for Toral's choice of the probability parameters $(p_0,p_1,p_2,p_3)=(1,4/25,4/25,7/10)$.  Computations were done using \textit{Mathematica 8} and a minor modification of the program displayed in Appendix A of \cite{EL12b}.  By Table~3 of \cite{EL12b}, $\mu_B<0$ for $3\le N \le 18$ except for $N=4,7,8$, so the Parrondo effect is present except in 22 of the 96 cases. Furthermore this calculation suggests that each mean $\mu_{[r,s]}$ converges to a limit as $N\to\infty$, although the rate of convergence is much slower than in the random-mixture case.

In fact, in \cite{EL12d} we proved that, under certain conditions (see Appendix A herein), $\mu_{[r,s]}$ converges to the same limit that $\mu_{(\gamma, 1-\gamma)}$, the mean profit of the randomly mixed game $\gamma A+(1-\gamma)B$, converges to, where $\gamma:=r/(r+s)$. The last row of Table~\ref{ex1} shows this limit. Tables~\ref{ex2} and \ref{ex3} analyze two other examples, a second point on the boundary of the unit cube and a point in the interior.  These are the same examples as the ones studied in \cite{EL12b}.

To illustrate the slower convergence rate, let us consider the example $p_0=1/10$, $p_1=p_2=3/5$, and $p_3=3/4$.  By $N=18$, $\mu_{[1,1]}$ matches its limiting value to only two significant digits.  On the other hand, by $N=19$ (the largest $N$ for which computations have been done in the random-mixture case), $\mu_B$ has stabilized to four significant digits and $\mu_{(1/2,1/2)}$ has stabilized to 11 significant digits.

For $N=18$, $[r,s]=[1,3]$, $[2,2]$, or $[3,1]$, and $p_0,p_1(=p_2),p_3\in(0,1)$, evaluation of the stationary distribution of $\bar{\bm P}_A^r\bar{\bm P}_B^s$ is computationally intensive because 32.4\% of the $59{,}059{,}225$ entries of this one-step transition matrix are nonzero. This is the reason we stopped with $N=18$. On the other hand, only 0.236\% of the entries of the one-step transition matrix $(r+s)^{-1}(r\bar{\bm P}_A+s\bar{\bm P}_B)$ are nonzero, so it is relatively easy to calculate its stationary distribution.  Furthermore, the faster convergence in the random-mixture case allows us to obtain six significant digits of $\lim_{N\to\infty}\mu_{(r/(r+s),s/(r+s))}$ using only the computed values for $N\le18$.

\begin{table}[ht]
\caption{\label{ex1}Assuming $(p_0,p_1,p_2,p_3)=(1,4/25,4/25,7/10)$, entries give $\mu_{[r,s]}$ to six significant digits
(in most cases) for $3\le N\le 18$ and six choices of $[r,s]$.  By \cite{EL12b}, $\mu_B<0$ for $3\le N\le18$ except for $N=4,7,8$, so the Parrondo effect is present except in 22 of the 96 ($=16\times6$) cases.  Also given, in the last row, is $\lim_{N\to\infty}\mu_{[r,s]}$ to six significant digits.\medskip}
\catcode`@=\active \def@{\hphantom{0}}
\catcode`#=\active \def#{\hphantom{$-$}}
\tabcolsep=.12cm
\begin{center}
\begin{footnotesize}
\begin{tabular}{ccccccc}
\noalign{\smallskip}
 $N$ & \multicolumn{6}{c}{$[r,s]$} \\
\noalign{\smallskip}
\hline
\noalign{\smallskip}
 & $[1,1]$ & $[1,2]$ & $[1,3]$ & $[2,1]$ & $[2,2]$ & $[3,1]$ \\
\noalign{\smallskip}
\hline
\noalign{\smallskip}
@3 & $-0.00695879$ & $-0.02748212$ & $-0.04021572$ & #0.00067249 & $-0.01487182$ & #0.00179203 \\  
@4 & #0.00877041 & #0.02345828 & #0.03569464 & #0.00352220 & #0.01011944 & #0.00244238 \\
@5 & #0.00466232 & #0.00501198 & #0.00434917 & #0.00320648 & #0.00465517 & #0.00240873 \\
@6 & #0.00497503 & #0.00590528 & #0.00513509 & #0.00325099 & #0.00498178 & #0.00241857 \\
@7 & #0.00496767 & #0.00621483 & #0.00637676 & #0.00326314 & #0.00497331 & #0.00242540 \\
@8 & #0.00494802 & #0.00604194 & #0.00599064 & #0.00327193 & #0.00495138 & #0.00243115 \\
@9 & #0.00493507 & #0.00598135 & #0.00588386 & #0.00327802 & #0.00493728 & #0.00243582 \\
10 & #0.00492347 & #0.00593756 & #0.00584200 & #0.00328237 & #0.00492494 & #0.00243961 \\
11 & #0.00491339 & #0.00589846 & #0.00578690 & #0.00328558 & #0.00491438 & #0.00244272 \\
12 & #0.00490464 & #0.00586697 & #0.00574489 & #0.00328800 & #0.00490531 & #0.00244529 \\
13 & #0.00489699 & #0.00584063 & #0.00571065 & #0.00328986 & #0.00489744 & #0.00244745 \\
14 & #0.00489026 & #0.00581820 & #0.00568131 & #0.00329133 & #0.00489056 & #0.00244927 \\
15 & #0.00488431 & #0.00579891 & #0.00565623 & #0.00329249 & #0.00488449 & #0.00245083 \\
16 & #0.00487900 & #0.00578213 & #0.00563452 & #0.00329343 & #0.00487912 & #0.00245217 \\
17 & #0.00487426 & #0.00576740 & #0.00561552 & #0.00329420 & #0.00487432 & #0.00245334 \\
18 & #0.00486999 & #0.00575438 & #0.00559876 & #0.00329483 & #0.00487001 & #0.00245437 \\
\noalign{\medskip}
$\infty$ & #0.00479089 & #0.00554084 & #0.00532972 & #0.00329853 & #0.00479089 & #0.00246903 \\
\noalign{\smallskip}
\hline
\end{tabular}
\end{footnotesize}
\end{center}
\end{table}

\begin{table}[ht]
\caption{\label{ex2}Assuming $(p_0,p_1,p_2,p_3)=(7/10,17/25,17/25,0)$, entries give $\mu_{[r,s]}$ to six significant digits (in most cases) for $3\le N\le18$ and six choices of $[r,s]$.   By \cite{EL12b}, $\mu_B<0$ for $3\le N\le18$ except for $N=3,5$, so the Parrondo effect is present except in 13 of the 96 cases.\medskip}
\catcode`@=\active \def@{\hphantom{0}}
\catcode`#=\active \def#{\hphantom{$-$}}
\tabcolsep=.12cm
\begin{center}
\begin{footnotesize}
\begin{tabular}{ccccccc}
\noalign{\smallskip}
 $N$ & \multicolumn{6}{c}{$[r,s]$} \\
\noalign{\smallskip}
\hline
\noalign{\smallskip}
 & $[1,1]$ & $[1,2]$ & $[1,3]$ & $[2,1]$ & $[2,2]$ & $[3,1]$ \\
\noalign{\smallskip}
\hline
\noalign{\smallskip}
@3 & #0.02224875 & #0.04081923 & #0.04914477 & #0.01124438 & #0.02978942 & #0.00756605 \\
@4 & #0.00771486 & #0.00053987 & $-0.00667489$ & #0.00804643 & #0.00719297 & #0.00668482 \\  
@5 & #0.00959590 & #0.00896406 & #0.00788233 & #0.00808935 & #0.00944220 & #0.00660145 \\
@6 & #0.00883530 & #0.00657582 & #0.00395457 & #0.00787821 & #0.00870056 & #0.00649756 \\
@7 & #0.00866607 & #0.00674014 & #0.00456448 & #0.00774894 & #0.00856648 & #0.00641956 \\
@8 & #0.00847909 & #0.00646246 & #0.00421207 & #0.00764628 & #0.00840151 & #0.00635716 \\
@9 & #0.00834684 & #0.00633835 & #0.00413293 & #0.00756557 & #0.00828507 & #0.00630657 \\
10 & #0.00824074 & #0.00622264 & #0.00402604 & #0.00750030 & #0.00819047 & #0.00626485 \\
11 & #0.00815521 & #0.00613328 & #0.00394982 & #0.00744652 & #0.00811356 & #0.00622992 \\
12 & #0.00808457 & #0.00605892 & #0.00388449 & #0.00740147 & #0.00804952 & #0.00620028 \\
13 & #0.00802529 & #0.00599680 & #0.00383017 & #0.00736319 & #0.00799540 & #0.00617485 \\
14 & #0.00797483 & #0.00594398 & #0.00378377 & #0.00733029 & #0.00794905 & #0.00615279 \\
15 & #0.00793135 & #0.00589853 & #0.00374381 & #0.00730171 & #0.00790890 & #0.00613348 \\
16 & #0.00789351 & #0.00585901 & #0.00370899 & #0.00727665 & #0.00787378 & #0.00611645 \\
17 & #0.00786027 & #0.00582433 & #0.00367839 & #0.00725451 & #0.00784280 & #0.00610132 \\
18 & #0.00783085 & #0.00579365 & #0.00365129 & #0.00723480 & #0.00781528 & #0.00608779 \\
\noalign{\medskip}
$\infty$ & #0.00734784 & #0.00529172 & #0.00320388 & #0.00689768 & #0.00734784 & #0.00584750 \\
\noalign{\smallskip}
\hline
\end{tabular}
\end{footnotesize}
\end{center}
\end{table}

\begin{table}[ht]
\caption{\label{ex3}Assuming $(p_0,p_1,p_2,p_3)=(1/10,3/5,3/5,3/4)$, entries give $\mu_{[r,s]}$ to six significant digits (in most cases) for $3\le N\le18$ and six choices of $[r,s]$.  By \cite{EL12b}, $\mu_B<0$ for $3\le N\le18$, so the Parrondo effect is present except in 2 of the 96 cases.\medskip}
\catcode`@=\active \def@{\hphantom{0}}
\catcode`#=\active \def#{\hphantom{$-$}}
\tabcolsep=.12cm
\begin{center}
\begin{footnotesize}
\begin{tabular}{ccccccc}
\noalign{\smallskip}
 $N$ & \multicolumn{6}{c}{$[r,s]$} \\
\noalign{\smallskip}
\hline
\noalign{\smallskip}
 & $[1,1]$ & $[1,2]$ & $[1,3]$ & $[2,1]$ & $[2,2]$ & $[3,1]$ \\
\noalign{\smallskip}
\hline
\noalign{\smallskip}
@3 & #0.0064599 & $-0.0113380$ & $-0.0305982$ & #0.00792562 & #0.0013604 & #0.00640450 \\ 
@4 & #0.0143185 & #0.0137926 & #0.0079878 & #0.00959254 & #0.0141773 & #0.00690030 \\ 
@5 & #0.0154167 & #0.0187545 & #0.0174399 & #0.00977539 & #0.0153785 & #0.00699045 \\
@6 & #0.0155020 & #0.0197053 & #0.0197681 & #0.00980464 & #0.0154761 & #0.00703029 \\
@7 & #0.0154423 & #0.0197654 & #0.0202208 & #0.00980993 & #0.0154198 & #0.00705342 \\
@8 & #0.0153749 & #0.0196547 & #0.0201990 & #0.00980937 & #0.0153555 & #0.00706817 \\
@9 & #0.0153182 & #0.0195309 & #0.0200777 & #0.00980683 & #0.0153015 & #0.00707811 \\
10 & #0.0152717 & #0.0194237 & #0.0199518 & #0.00980357 & #0.0152575 & #0.00708508 \\
11 & #0.0152334 & #0.0193346 & #0.0198414 & #0.00980014 & #0.0152212 & #0.00709014 \\
12 & #0.0152014 & #0.0192605 & #0.0197479 & #0.00979677 & #0.0151908 & #0.00709392 \\
13 & #0.0151742 & #0.0191982 & #0.0196688 & #0.00979356 & #0.0151649 & #0.00709679 \\
14 & #0.0151508 & #0.0191452 & #0.0196013 & #0.00979055 & #0.0151427 & #0.00709903 \\
15 & #0.0151306 & #0.0190994 & #0.0195432 & #0.00978777 & #0.0151233 & #0.00710080 \\
16 & #0.0151128 & #0.0190596 & #0.0194927 & #0.00978519 & #0.0151064 & #0.00710221 \\
17 & #0.0150971 & #0.0190246 & #0.0194483 & #0.00978280 & #0.0150913 & #0.00710337 \\
18 & #0.0150831 & #0.0189936 & #0.0194090 & #0.00978060 & #0.0150779 & #0.00710431 \\
\noalign{\medskip}
$\infty$ & #0.0148448 & #0.0184829 & #0.0187645 & #0.00973219 & #0.0148448 & #0.00710942 \\
\noalign{\smallskip}
\hline
\end{tabular}
\end{footnotesize}
\end{center}
\end{table}

\section{Conclusions}

We considered the spatially dependent, or cooperative, Parrondo games of Toral \cite{T01}, which assume $N\ge3$ players arranged in a circle, and in which the win probability for a player depends on the status of the player's two nearest neighbors.  There are three games, game $A$ without spatial dependence, game $B$ with spatial dependence, and the nonrandom periodic pattern $A^rB^s$, or $[r,s]$, with $r,s\ge 1$.  The model is described by a Markov chain with $2^N$ states and parameters $p_0,p_1,p_2,p_3$.  To maximize the value of $N$ for which exact computations are feasible, we assumed $p_1=p_2$ and regarded states as equivalent if they are equal after a rotation and/or reflection of the players.  This allowed us to compute the mean profits per turn, $\mu_{[r,s]}$, to the ensemble of $N$ players for $3\le N\le18$ and $r+s\le4$ and several choices of the parameter vector $(p_0,p_1,p_3)$, including Toral's choice.  The results provide numerical evidence, but not a proof, that $\mu_{[r,s]}$ converges as $N\to\infty$ and that the Parrondo effect (i.e., $\mu_B\le0$ and $\mu_{[r,s]}>0$) persists for all $N$ sufficiently large for a set of parameter vectors having nonzero volume.  As we noted in \cite{EL12b}, this suggests that the spatially-dependent version of Parrondo's paradox is a robust phenomenon that remains present in the thermodynamic limit.

This is the main conclusion, but there are several other noteworthy findings.  We have shown that the sequence of profits to the ensemble of $N$ players obeys the strong law of large numbers.  This is relevant to defining what is meant by a winning, losing, and fair game.  For $N=3,4,5$ and $r+s\le3$, explicit formulas for $\mu_B$ and $\mu_{[r,s]}$ are available, so with the help of computer graphics, we have demonstrated that one can visualize the Parrondo region, the region in the three-dimensional parameter space in which the Parrondo effect appears.  There is also an anti-Parrondo region, and we have shown that it is symmetric with the Parrondo region, as might be expected.

We have restricted our attention to the one-dimensional version of the model, but our methods may have applicability to the two-dimensional version, already investigated by Mihailovi\'c and Rajkovi\'c \cite{MR06} using computer simulation.  In a separate paper \cite{EL12d} we prove, under conditions that can likely be improved (see Appendix A below), the convergence of $\mu_{[r,s]}$ as $N\to\infty$, and we express the limit in terms of a spin system on the one-dimensional integer lattice.

\section*{Acknowledgments}

The work of S. N. Ethier was partially supported by a grant from the Simons Foundation (209632).  It was also supported by a Korean Federation of Science and Technology Societies grant funded by the Korean Government (MEST, Basic Research Promotion Fund).  The work of J. Lee was supported by the Basic Science Research Program through the National Research Foundation of Korea (NRF) funded by the Ministry of Education, Science and Technology (2012-0004434).

\section*{Appendix A}

We have stated (referring to \cite{EL12c} and \cite{EL12d} for details) that, under certain conditions on $p_0,p_1,p_2,p_3$ and $[r,s]$, the means $\mu_B$ and $\mu_{[r,s]}$ converge as $N\to\infty$.  We would like to briefly describe the conditions and identify the limits.

The limits depend on the spin system in the infinite product space $\{0,1\}^{\bf Z}$ (${\bf Z}$ being the set of integers) with generator of the form
\begin{equation}\label{generator}
(\mathscr{L}_Bf)(\bm x):=\sum_{i\in{\bf Z}}c_i(\bm x)[f(\bm x^i)-f(\bm x)],
\end{equation}
where
\begin{equation}\label{flip rates}
c_i(\bm x):=\begin{cases}p_{m_i(\bm x)}&\text{if $x_i=0$,}\\ q_{m_i(\bm x)}&\text{if $x_i=1$,}\end{cases}
\end{equation}
and
$m_i(\bm x):=2x_{i-1}+x_{i+1}$ and $q_m:=1-p_m$ for $m=0,1,2,3$.  Intuitively, a ``spin,'' 0 or 1, is associated with each ``site'' $i\in{\bf Z}$, and the spin is ``flipped'' at site $i$ at exponential rate $c_i(\bm x)$.  If the sum in (\ref{generator}) were finite, this would define a pure-jump Markov process, but because the sum is infinite, it defines a Markov process with instantaneous states.  See Liggett \cite{L85} for a thorough account of spin systems.

We have shown in \cite{EL12c} that the spin system with generator $\mathscr{L}_B$ is ergodic (i.e., it has a unique stationary distribution and the distribution at time $t$ converges weakly to that stationary distribution as $t\to\infty$, regardless of the initial distribution) if at least one of the following four conditions is satisfied:
\begin{equation}\label{cond-a}
\max(|p_0-p_1|,|p_2-p_3|)+\max(|p_0-p_2|,|p_1-p_3|)<1;
\end{equation}
\begin{equation}\label{cond-b}
0<\min(p_0,p_3)\le \min(p_1,p_2)\le\max(p_1,p_2)\le \max(p_0,p_3)<1;
\end{equation}
\begin{equation}\label{cond-c}
\;\;\max(p_1,p_2,p_3,p_1+p_2-p_3)-p_3<p_0/2<\min(p_1,p_2,p_3,p_1+p_2-p_3);
\end{equation}
\begin{equation}\label{cond-d}
p_0,p_1,p_2,p_3\in(2\overline{p}-1,2\overline{p})\cap(0,1),\quad \overline{p}:=(p_0+p_1+p_2+p_3)/4.
\end{equation}
Under the assumption that $p_1=p_2$, the (three-dimensional) volumes of the regions described by (\ref{cond-a})--(\ref{cond-d}) are respectively 7/12, 1/3, 7/32, and 2/3.  The volume of the union of these four regions is 3323/4032, representing about 82.4\% of the parameter space.  Of our three examples (Tables \ref{ex1}--\ref{ex3}), only the third one belongs to this union.

It is shown in \cite{EL12c} that, if the spin system with generator $\mathscr{L}_B$ is ergodic with unique stationary distribution $\pi_B$, then the mean $\mu_B$ converges as $N\to\infty$ to
\begin{equation}\label{limit}
2\sum_{u=0}^1\sum_{v=0}^1(\pi_B)_{-1,1}(u,v)p_{2u+v}-1,
\end{equation}
where $(\pi_B)_{-1,1}$ denotes the $-1,1$ two-dimensional marginal of $\pi_B$.

If $p_0=p_1=p_2=p_3=1/2$, we denote $\mathscr{L}_B$ by $\mathscr{L}_A$.  A similar argument shows that, if $0<\gamma<1$ and the spin system with generator $\gamma\mathscr{L}_A+(1-\gamma)\mathscr{L}_B$ is ergodic with unique stationary distribution $\pi_C$, then, for the random mixture $C:=\gamma A+(1-\gamma)B$, the mean $\mu_C$ converges as $N\to\infty$ to (\ref{limit}) but with $\pi_B$ replaced by $\pi_C$ and each $p_m$ replaced by $\hat{p}_m:=\gamma(1/2)+(1-\gamma)p_m$.  Notice that $\gamma\mathscr{L}_A+(1-\gamma)\mathscr{L}_B$ is just (\ref{generator}) (with (\ref{flip rates})) but with each $p_m$ replaced by $\hat{p}_m$ (and each $q_m$ replaced by $\hat{q}_m:=1-\hat{p}_m$), so the ergodicity condition holds if at least one of (\ref{cond-a})--(\ref{cond-d}) holds with each $p_m$ replaced by $\hat{p}_m$.

Finally, it is shown in \cite{EL12d} that, if the assumption of the preceding paragraph holds with $\gamma:=r/(r+s)$, then the mean $\mu_{[r,s]}$ converges as $N\to\infty$ to the limit of the preceding paragraph with $\gamma:=r/(r+s)$.


\begin{thebibliography}{00}

\bibitem{T01}R. Toral, Cooperative Parrondo games,  \textit{Fluct. Noise Lett.} \textbf{1} (2001) L7--L12.

\bibitem{HA02}G. P. Harmer and D. Abbott, A review of Parrondo's paradox,  \textit{Fluct. Noise Lett.} \textbf{2} (2002) R71--R107.

\bibitem{A10}D. Abbott, Asymmetry and disorder: A decade of  Parrondo's paradox,  \textit{Fluct. Noise Lett.} \textbf{9} (2010) 129--156.

\bibitem{EL12b}S. N. Ethier and J. Lee, Parrondo games with spatial dependence, \textit{Fluct. Noise Lett.} \textbf{11} (2012), to appear,
\url{http://arxiv.org/abs/1202.2609}.

\bibitem{EL12c}S. N. Ethier and J. Lee, Parrondo games with spatial dependence and a related spin system, \url{http://arxiv.org/abs/1203.0818}.

\bibitem{EL12d}S. N. Ethier and J. Lee, Parrondo games with spatial dependence and a related spin system, II, \url{http://arxiv.org/abs/1206.xxxx}.

\bibitem{HA99}G. P. Harmer and D. Abbott, Parrondo's paradox, \textit{Statist. Sci.} \textbf{14} (1999) 206--213.

\bibitem{AP92}A. Ajdari and J. Prost, Drift induced by a spatially periodic potential of low symmetry: Pulsed dielectrophoresis, \textit{C. R. Acad. Sci., S\'er. 2} \textbf{315} (1992) 1635--1639.

\bibitem{EZ00}S. B. Ekhad and D. Zeilberger, Remarks on the Parrondo paradox,  \textit{Personal J. of Ekhad and Zeilberger} (2000) \url{http://www.math.rutgers.edu/~zeilberg/pj.html}

\bibitem{VW01}D. Velleman and S. Wagon, Parrondo's paradox, \textit{Math. Educ. Res.} \textbf{9} (2001) 85--90.

\bibitem{KJ03}R. J. Kay and N. F. Johnson, Winning combinations of history-dependent games, \textit{Phys. Rev. E} {\bf 67} (2003) 056128.

\bibitem{B04}E. Behrends, The mathematical background of Parrondo's paradox, \textit{Proc. SPIE Noise in Complex Systems and Stochastic Dynamics II}, Maspalomas, Spain, Ed: Z. Gingl, \textbf{5471} (2004) 510--517.

\bibitem{EL09}S. N. Ethier and J. Lee, Limit theorems for Parrondo's paradox,  \textit{Electron. J. Probab.} \textbf{14} (2009) 1827--1862.

\bibitem{D08}L. Dinis, Optimal sequence for Parrondo games, \textit{Phys. Rev. E} \textbf{77} (2008) 021124.

\bibitem{T02}R. Toral, Capital redistribution brings wealth by Parrondo's paradox,  \textit{Fluct. Noise Lett.} \textbf{2} (2002) L305--L311.

\bibitem{EL12a}S. N. Ethier and J. Lee, Parrondo's paradox via redistribution of wealth, \textit{Electron. J. Probab.} \textbf{17} (2012) (20) 1--21.

\bibitem{MR03}Z. Mihailovi\'c and M. Rajkovi\'c, One dimensional asynchronous cooperative Parrondo's games,  \textit{Fluct. Noise Lett.} \textbf{3} (2003) L389--L398.

\bibitem{PHA00}J. M. R. Parrondo, G. P. Harmer, and D. Abbott, New paradoxical games based on Brownian ratchets, \textit{Phys. Rev. Lett.} \textbf{85} (2000) 5226--5229.

\bibitem{MR06}Z. Mihailovi\'c and M. Rajkovi\'c, Cooperative Parrondo's games on a two-dimensional lattice,  \textit{Phys. A} \textbf{365} (2006) 244--251.

\bibitem{L85}T. M. Liggett, \textit{Interacting Particle Systems} (Springer-Verlag, New York, 1985).

\end{thebibliography}
\end{document}